\documentclass{amsart}
\usepackage[dvips]{graphicx}
\usepackage{amscd}
\usepackage{amsmath}
\usepackage{amsxtra}
\usepackage{amsfonts}
\usepackage{amssymb}

\newtheorem{theorem}{Theorem}[section]
\newtheorem{corollary}[theorem]{Corollary}

\newtheorem{proposition}[theorem]{Proposition}
\newtheorem{assumption}[theorem]{Assumption}

\theoremstyle{definition}
\newtheorem{definition}[theorem]{Definition}
\newtheorem{remark}[theorem]{Remark}
\newtheorem{problem}{Problem}

\newtheorem{example}[theorem]{Example}
\theoremstyle{remark}

\renewcommand{\theclaim}{\textup{\theclaim}}

\newtheorem*{acknowledgements}{Acknowledgements}

\numberwithin{equation}{section}

\def\openone%{\hbox{\upshape \small1\kern-3.3pt\normalsize1}}

{\mathchoice

{\hbox{\upshape \small1\kern-3.3pt\normalsize1}}

{\hbox{\upshape \small1\kern-3.3pt\normalsize1}}

{\hbox{\upshape \tiny1\kern-2.3pt\SMALL1}}

{\hbox{\upshape \Tiny1\kern-2pt\tiny1}}}

\makeatletter

\newbox\ipbox

\newcommand{\ip}[2]{\left\langle #1\, , \,#2\right\rangle}
\newcommand{\diracb}[1]{\left\langle #1\mathrel{\mathchoice

{\setbox\ipbox=\hbox{$\displaystyle \left\langle\mathstrut
#1\right.$}

\vrule height\ht\ipbox width0.25pt depth\dp\ipbox}

{\setbox\ipbox=\hbox{$\textstyle \left\langle\mathstrut
#1\right.$}

\vrule height\ht\ipbox width0.25pt depth\dp\ipbox}

{\setbox\ipbox=\hbox{$\scriptstyle \left\langle\mathstrut
#1\right.$}

\vrule height\ht\ipbox width0.25pt depth\dp\ipbox}

{\setbox\ipbox=\hbox{$\scriptscriptstyle \left\langle\mathstrut
#1\right.$}

\vrule height\ht\ipbox width0.25pt depth\dp\ipbox}

}\right. }

\newcommand{\dirack}[1]{\left. \mathrel{\mathchoice

{\setbox\ipbox=\hbox{$\displaystyle \left.\mathstrut
#1\right\rangle$}

\vrule height\ht\ipbox width0.25pt depth\dp\ipbox}

{\setbox\ipbox=\hbox{$\textstyle \left.\mathstrut
#1\right\rangle$}

\vrule height\ht\ipbox width0.25pt depth\dp\ipbox}

{\setbox\ipbox=\hbox{$\scriptstyle \left.\mathstrut
#1\right\rangle$}

\vrule height\ht\ipbox width0.25pt depth\dp\ipbox}

{\setbox\ipbox=\hbox{$\scriptscriptstyle \left.\mathstrut
#1\right\rangle$}

\vrule height\ht\ipbox width0.25pt depth\dp\ipbox}

} #1\right\rangle}

\newcommand{\bz}{\mathbb{Z}}

\newcommand{\br}{\mathbb{R}}
\newcommand{\bc}{\mathbb{C}}
\newcommand{\bt}{\mathbb{T}}
\newcommand{\bn}{\mathbb{N}}

\def\blfootnote{\xdef\@thefnmark{}\@footnotetext}

\renewcommand{\mod}{\operatorname{mod}}

\input xy
\xyoption{all}
\usepackage{amssymb}

\def\-{^{-1}}

\begin{document}
\title[Fourier series on fractals]{Fourier series on fractals: a parallel with wavelet theory}
\author{Dorin Ervin Dutkay}
\blfootnote{Research supported in part by a grant from the National Science Foundation DMS-0704191}
\address{[Dorin Ervin Dutkay] University of Central Florida\\
	Department of Mathematics\\
	4000 Central Florida Blvd.\\
	P.O. Box 161364\\
	Orlando, FL 32816-1364\\
U.S.A.\\} \email{ddutkay@mail.ucf.edu}

\author{Palle E.T. Jorgensen}
\address{[Palle E.T. Jorgensen]University of Iowa\\
Department of Mathematics\\
14 MacLean Hall\\
Iowa City, IA 52242-1419\\}\email{jorgen@math.uiowa.edu}
\thanks{} 
\dedicatory
{Dedicated to the memory of Thomas P. Branson.}

\begin{abstract}
 We study orthogonality relations for Fourier frequencies and complex
exponentials in Hilbert spaces $L^2(\mu)$ with measures $\mu$ arising from
iterated function systems (IFS). This includes equilibrium measures in
complex dynamics. Motivated by applications, we draw parallels between
analysis of fractal measures on the one hand, and the geometry of wavelets
on the other.
   We are motivated by spectral theory for commuting partial differential
operators and related duality notions. While stated initially for bounded
and open regions in $\br^d$, they have since found reformulations in the theory
of fractals and wavelets. We include a historical sketch with questions from
early operator theory.
\end{abstract}
\maketitle \tableofcontents
\section{Introduction}\label{intro}

In this paper we give an overview (supplemented with new results and open problems) of recent research on orthogonality relations for Fourier frequencies, or rather the study of complex exponentials, see details below. This will be used in drawing up parallels between analysis of fractal measures on the one hand, and the geometry of wavelets on the other.

      This allows us to further include new results here involving non-linear complex systems.

       Much of our motivation comes from a problem dating back to the use of von Neumann's Spectral Theorem in the study of commuting partial differential operators. This led to a formulation of a number of duality notions. While they started for bounded and open regions in $\br^d$, many of the questions have found reformulations in the theory of fractals and wavelets. At the end of our paper, we include a historical section making connections to questions from the beginning of operator theory.

It is the purpose of our paper to draw attention to these interconnections. For the benefit of the reader, we composed in the Appendix a personal and historical section dealing with spectrum-tile duality the way it was first formulated by I.E. Segal and B. Fuglede.

     Specifically, we study $L^2(\mu)$-Hilbert spaces with an emphasis on measures $\mu$  arising as equilibrium measures for iterated function systems (IFSs) in $\br^d$, i.e., for finite contractive function systems in $\br^d$, so in $d$ real dimensions.

        One of the applications of these IFSs, and their spectral theory, is to image processing, see e.g., \cite{Bar06} and \cite{DuJo06a}. The viewpoint in \cite{Bar06} is that a variety of images, e.g., large bit-size satellite photographs, may be condensed into IFS-codes. The trick is that such large-scale images can be shown to allow effective IFS-representations; and that the IFSs are encoded with a very small set of parameters. This of course is a feature that IFSs share with wavelets, see also \cite{DuJo06a}.
        
        The paper is structured as follows: in Section \ref{mraa} we present the parallel between multiresolution theory and the theory of orthonormal Fourier series on fractals.  In Section \ref{open} we formulate several open problems related mainly to the fractal context. In the Appendix we describe the history and motivation of the subject. 
        At several points in the paper we will make connections to iteration of complex rational maps.

\subsection{Complex systems}\label{comp}     

 While affine iterated function systems (IFSs) are traditionally defined for a fixed and finite set of transformations $S = (\tau_i)$ in $\br^d$, it is often helpful to realize them by embeddings into compact quotients. Specifically, starting with a particular system $S$, to get a fractal iteration limit $X$ (see Definitions \ref{defaff}, and \ref{definvmea}), one usually imposes suitable contractivity properties on the maps $\tau_i$ making up a particular system $S$. If the mappings $\tau_i$ are affine it is occasionally possible to embed an associated fractal $X$ in a compact torus $\bt^d = \br^d/\Gamma$ where $\Gamma$ is a suitably chosen lattice in $\br^d$, e.g., $\Gamma= \bz^d$.

       This makes an intriguing connection to complex dynamics \cite{Bea91, Bea92} and Hilbert spaces of analytic functions which we address below, and again in section 3. But the connection to compact subsets in the complex plane opens the door to the use of Hilbert space bases in the study of Julia sets derived from the iteration of polynomial mappings $z\mapsto R(z)$ in the complex plane, see \cite{Bea91}.

   A main focus here will be on affine fractal measures $\mu$, but also including measures $\mu$ which are supported on Julia sets for certain rational mappings of one complex variable.

         In the  case of measures $\mu$ supported on compact subsets $X$ in the complex plane $\bc$, orthogonality relations for Fourier frequencies $z^n$ may be studied in terms of the moments, i.e., $M_n(\mu) :=\int z^n\, d\mu(z)$   for $n = 0, 1,\dots$.

      The following example illustrates both orthogonality and linear relations within the set of monomials $z^n$ in $L^2(\mu)$. Our discussion includes measures $\mu$; both when $L^2(\mu)$ is finite and infinite-dimensional. The infinite-dimensional cases often arise as IFS-limits.

      Indeed, this yields attractive classes of measures $\mu$ with compact support in the complex plane $\bc$ such that $L^2(\mu)$ is infinite-dimensional, and such that there are non-trivial linear relations governing the monomials $\{z^n\, |\, n \in \bn_0\}$ restricted to the support of $\mu$: Specifically, when each $z^n$ is viewed as a vector in the Hilbert space $L^2(\mu)$.

       This may be seen by a choice of certain fractal measures $\mu$ having compact support in $\bc$, and chosen in such a manner that a family of Gram-Schmidt polynomials $p_k(z)$, $k = 0, 1, 2,\dots$ are in fact monomials $z^n$, but with $n$ restricted to an infinite lacunary progression in $\bn_0$; a rather sparse subset of $\bn_0$.

       Using a construction of Jorgensen-Pedersen  \cite{JoPe98} it is possible to get the lacunary progression $z^n$ making up an ONB in $L^2(\mu)$ for values of $n$ arranged to be sparsely distributed in $\bn_0$ in such a manner that gaps-sizes are powers of $4$.

    We refer to \cite{JoPe98} for a complete discussion of this class of examples. But for the benefit of readers, a summary is included here. The details of the construction further helps make the connection between the setting of \cite{JoPe98} and the present somewhat different set of questions.

     Details: Arrange a $2\pi$-rescaling of the circumference of the circle $\bt$ in $\bc$, so it scales to one. Now construct an IFS measure $\mu$ with a certain compact and fractal support $X$ contained in $\bt$ as follows: Subdivide the circumference into $4$ equal parts, and define the two mappings $\{\tau_0, \tau_1\}$ that select every second of the $4$ subintervals. In a linear scale, we may take $\tau_0(x) = x/4$, and $\tau_1(x) = (x + 2)/4$. Pick probabilities $(1/2, 1/2)$. Then by Hutchinson's theorem \cite{Hut81}, this system has a unique equilibrium measure $\mu$, and the theorem in \cite{JoPe98} states that the complex monomials $\{z^n\, |\, n = 0, 1, 4, 5, 16, 17, 20, 21, \dots\}$ form an ONB in $L^2(\mu)$.
     
    The following curious conclusion follows:  
   \begin{corollary}\label{cor1_1} {\it Every} function $f$ in $L^2(\mu)$, for $\mu$ the scale-4 measure, is the restriction to the scale-4 Cantor set $X = \mbox{supp}(\mu)$ of a Hardy $H^2$-function, say $F$.
   \end{corollary}

Specifically, we identify $F$ with an analytic function on $\{z\,|\, |z| < 1\}$, and with its a.e. non-tangential limit to the boundary circle $\bt$ (the one-torus). And we further identify $X$ with a subset of $\bt$ as described above.

        \begin{remark} The corresponding result is clearly not true for $L^2(\bt, \mbox{Haar})$, i.e.,  if the Cantor measure $\mu$ on $\bt$ is replaced with the Haar measure; or equivalently with the restricted Lebesgue measure to a period interval.
\end{remark}

      In contrast, it was proved in \cite{JoPe98} that the analogous construction with the middle-third Cantor set yields no $L^2(\mu)$-orthogonal subsets of $\{z^n\, |\, n \in \bn_0\}$ of cardinality more than two; in fact there cannot be more than two orthogonal Fourier frequencies, even allowing real (non-integer) frequencies.

     In the scale-$4$ fractal case, it can be shown \cite{JoPe98} that the Hausdorff dimension of $\mu$ and its support $X$ is $1/2$. Hence, as a subset of $\bt$, $X$ is ``thin'', but still with the cardinality of the continuum.

     Orthogonality: The sparse distribution in $\bn_0$ of the exponents in the ONB for $L^2(\mu)$ for the scale-$4$ case is as follows: The values of $n$ which turn the monomials $z^n$ into an ONB in $L^2(\mu)$ are represented by the following set $\Lambda$ of all finite sums $\lambda=j_0 + j_1 4 + j_2 4^2 + j_3 4^3 + \dots$  where each $j_k$ is a $0-1$ bit, i.e., each $j_k$ is chosen in $\{0, 1\}$.

     For related constructions, see also \cite{DuJo06a, DuJo07a, DuJo07b}.

      Corollary \ref{cor1_1} is a definite statement about the specific polynomial $z \mapsto z^4$ and its iterations. It relates to functions $F(z)$ that are analytic in the interior of the filled Julia set; in this case the closed disk  $\{ z \in \bc \,|\, |z| \leq 1\}$.

\begin{definition}
Let $R = R(z)$ be a polynomial. We will view $R$ as an endomorphism in the complex plane, or in the one-point compactification, the Riemann sphere. The $n$-fold iteration of $z\mapsto R(z)$ will be denoted $R^n$, i.e., inductively,  $R^{n + 1}(z) := R(R^n(z))$, for $n = 1, 2,\dots$. The filled Julia set $Y(R)$ of $R$ consists of the points with bounded orbits, i.e.,  $Y(R) := \{ z \,|\, R^n(z)\mbox{ is bounded }\}$. The Fatou set $F(R)$ is the largest open set in which the sequence $(R^n)$ is normal, and the Julia set $J(R)$ is the complement of $F(R)$.
     It is known \cite{Bea91} that $J(R)$ is the boundary of the filled Julia set $Y(R)$, when $Y(R)$ is not empty.
     \end{definition}
      
      It would be interesting to prove an analogous result for $z\mapsto z^2 + c$, or for $z\mapsto R(z)$ where $R$ is some polynomial. Then the question below will be about functions $F(z)$ which are analytic in the interior of the filled Julia set for the particular polynomial $z\mapsto R(z)$.

Recall, in general, the boundary of the filled Julia set for $R$ is the Julia set $J(= J(R))$ coming from the dynamics of $R$; and moreover that the IFS-Brolin measure $\mu_R$ is supported on $J(R)$.

Let $z\mapsto R(z)$ be given, and let's introduce the Brolin measure \cite{Bro65} $\mu_R$ on $J(R)$. Let $H^2(J(R))$ be the functions in
$L^2(J(R), \mu_R)$ which have analytic continuation to the interior of the filled Julia set of $R$.

\begin{definition}\label{defk2}
      Let $z\mapsto R(z)$ be a given polynomial. Let $P$ be a finite set of distinct polynomials each of degree less than the degree of $R$. Let $K^2(R, P)$ be a space of functions $F$ in $H^2(J(R))$. We say that $F$ is in $K^2(R, P)$  iff there are functions $F_p(z)$ in $H^2(J(R))$ indexed by $P$ such that
 \begin{equation}\label{eqdz}
F(z) = \sum_{p \in P} p(z) F_p(R(z)).\end{equation} It is easy to see that each $K^2(R, P)$ is a closed subspace in $H^2(J(R))$, so in particular it is a Hilbert space.
\end{definition}

\begin{remark}
The finite family $P$ that enters \eqref{eqdz} is a family of generalized filters. It depends on the particular polynomial $z\mapsto R(z)$, and we can expect to find solutions to \eqref{eqdz} in Definition \ref{defk2} from the kind of representations of the Cuntz algebras the authors studied in their work on wavelets on fractals, \cite{DuJo06a}.
\end{remark}

       A refinement of the conclusion in Corollary \ref{cor1_1} is this: For $R(z) := z^4$, set $P := \{1, z\}$. Hence the functions $F$ in $K^2(R, P)$ have the representation $F(z) = F_0(z^4) + z F_1(z^4)$ for a pair of functions $F_0$ and $F_1$ in the Hardy space $H^2$. Our Corollary \ref{cor1_1} above then states that there is an isometric isomorphism of $L^2(J(R), \mu_R)$ with the Hilbert space $K^2(R, P)$.

 To illustrate Corollary \ref{cor1_1} and \eqref{eqdz}, we supply the proof details in the case of $R(z):=z^4$, and $P=\{1,z\}$.
 
 Since the spectrum for $\mu$ is the set $\Lambda=\{\sum_{k=0}^nj_k4^k\,|,j_k\in\{0,1\}\}\subset\bn_0$, the family $\{e_\lambda\,|\,\lambda\in\Lambda\}$ is an ONB in $L^2(\mu)=L^2(X,\mu)$ with $X\subset\bt$, and we get the following $L^2(\mu)$-orthogonal expansion $f\mapsto F_f$
 \begin{equation}\label{eqdzdz}
 F_f(z)=\sum_{\lambda\in\Lambda}c_\lambda z^\lambda
 \end{equation}
  valid for all $f\in L^2(\mu)$. By this we mean that every $f\in L^2(\mu)$ is the boundary representation of the function $F_f(z)$ in \eqref{eqdzdz}, analytic in the disk $D:=\{z\in\bc\,|\, |z|<1\}$. 
  
  Specifically, \eqref{eqdzdz} is absolutely convergent in $D$; and applying Schwarz to \eqref{eqdzdz} we get
  $$|F_f(z)|\leq\|f\|_{L^2(\mu)}\frac{1}{\sqrt{1-|z|^2}},\quad(z\in D).$$
  
  To get the desired representation
  \begin{equation}\label{eqdz3}
  F_f(z)=F_0(z^4)+zF_1(z^4)
  \end{equation}
  from Definition \ref{defk2}, note first that $\Lambda=4\Lambda\cup(1+4\Lambda)$.
  
  Substituting this into \eqref{eqdzdz} then yields
  $$F_f(z)=\sum_{\lambda_1\in\Lambda}c_{4\lambda_1}(z^4)^{\lambda_1}+z\sum_{\lambda_2\in\Lambda}c_{4\lambda_2+1}(z^4)^{\lambda_2},$$
  which is the desired formula \eqref{eqdz3}.
  
  Since the functions $e_\lambda(x)=z^\lambda$ form an ONB in $L^2(\mu)$, the representation in \eqref{eqdz3} is an orthogonal expansion, i.e., we have
  $$\|f\|_{L^2(\mu)}^2=\|F_f\|_{H^2}^2=\|F_0\|_{H^2}^2+\|F_1\|_{H^2}^2.$$

\section{Multiresolution wavelets and Iterated Function Systems}\label{mraa}

 The purpose of this section is to present some key results regarding geometry and spectral theory for wavelets and for the fractal measures which arise from iterated function systems (IFSs). Our focus is on those wavelets which come from multiresolutions (MRAs), see subsection \ref{scal} below. This includes wavelets on fractals, as presented in the paper \cite{DuJo06a} by the co-authors.
      
      As mentioned in the Introduction, the IFSs include dynamical systems defined from a finite set of affine and contractive mappings in $\br^d$, or from the branches of inverses of complex polynomials, or of rational mappings in the complex plane.

       In terms of signal processing, what the two have in common, wavelets and IFSs, is that large scale data may be compressed into a few functions or parameters. In the case of IFSs, only a few matrix entries are needed, and a finite set of vectors in $\br^d$ must be prescribed. As is shown in \cite{Bar06}, this can be turned into effective codes for large images. Similarly (see \cite{Jor06}) discrete wavelet algorithms can be applied to digital images and to data mining. The efficiency in these application lies in the same fact: The wavelets may be represented and determined by a small set of parameters; a choice of scaling matrix and of masking coefficients, i.e., the coefficients $(a_k)$ in the scaling identity \eqref{eqscaling} below.

      The main characters in the chapter are attractors (section \ref{att}) understood in the general sense so they include both wavelets and IFSs. In section \ref{mra}, we introduce multiresolution analyses (MRAs) in the form of scales of nested subspaces in a fixed Hilbert space, which may be $L^2(\br^d)$, or some other Hilbert space derived from a computational setup, for example for iteration of rational functions in the complex plane (subsection \ref{comp}).

      Our scaling identity is discussed in section \ref{scal}, but again in a setting which includes both the traditional wavelets, and their fractal variants. In subsections \ref{ortho} - \ref{lawton}, we study orthogonality relations from the view point of a transfer operator; an infinite-dimensional operator analogous to the Perron-Frobenius operator from matrix analysis .

      It turns out that there is a Perron-Frobenius operator in wavelet and fractal theory which encodes orthogonality relations. We call it ``the transfer operator'', or ``the Perron-Frobenius-Ruelle operator'', but other names may reasonably be associated with it, see e.g., \cite{BrJo02}. In the context of wavelets, the operator was studied in \cite{Law91} . It will be specified by a chosen filter function $m_0$, and we denote the associated operator  by $R_{m_0}$.

       Sections \ref{cohen} and \ref{super} are about applications of the transfer operator. The term ``super-wavelet'' in section \ref{super} refers to a super-structure. It is proved that non-orthogonal wavelets, satisfying a degree of stability (Parseval frames), can be turned into orthogonal wavelets in an enlarged Hilbert space. This is done in section \ref{super}. Motivations are from signal/image processing. By removing correlations between frequency bands, or between nearby pixels, redundancies are removed, and reduction in data size, and data storage, can thereby be achieved.

\begin{definition}\label{defspectral}
For $\lambda\in\br^d$, let $e_\lambda(x)=e^{2\pi i\lambda\cdot x}$, $x\in\br^d$.
A measure $\mu$ on $\br^d$ is said to be a {\it spectral measure} if there exists a set $\Lambda\subset\br^d$ such that the family $\{e_\lambda\,|\,\lambda\in\Lambda\}$ forms an orthonormal basis for $L^2(\mu)$. In this case, the set $\Lambda$ is called a {\it spectrum} of the measure $\mu$.
\end{definition}

We are interested in the question: what spectral measures can be constructed, other than the Lebesgue measure on an interval? A surprising answer will be offered by affine iterated function systems.

\begin{definition}\label{defaff}
Let $A$ be a $d\times d$ expansive integer matrix. We say that a matrix is {\it expansive} if all its eigenvalues have absolute value strictly greater than $1$. 

Let $B$ be a finite set of points $\bz^d$, of cardinality $|B|=:N$.

For each $b\in B$, we define the following affine maps on $\br^d$,
\begin{equation}\label{eqtaub}
\tau_b(x)=A^{-1}(x+b),\quad(x\in\br^d).
\end{equation}
The family of functions $(\tau_b)_{b\in B}$ is called an {\it affine iterated function system (affine IFS)}. 
\end{definition}

\subsection{The attractor}\label{att}

Since the matrix $A$ is expansive, there is a norm on $\br^d$ such that all the maps $\tau_b$ are contractive. Therefore the following general theorem for contractive iterated applies:
\begin{theorem}\label{thattr}\cite{Hut81}
Let $(\tau_i)_{i=1}^N$ be a system of $N$ contractive maps on a complete metric space $X$. Then there is a unique compact subset $K\subset X$ such that
\begin{equation}\label{eqattr}
K=\bigcup_{i=1}^N \tau_i(K).
\end{equation}
\end{theorem}

\begin{definition}\label{defattr}
The compact set $K$ defined by equation \eqref{eqattr} in Theorem \ref{thattr} is called the {\it attractor} of the {\it iterated function system} $(\tau_i)_{i=1}^N$.
\end{definition}

\begin{example}\label{excantor}\textup{(The middle third Cantor set)}
Let us consider the middle third Cantor set $\mathbf C$. This set can be regarded as the attractor of an affine iterated function system in $\br$. So $d=1$. The matrix is $A=3$, the set $B$ is $B=\{0,2\}$. Therefore the maps are 
$\tau_0(x)=\frac{x}{3}$ and $\tau_2(x)=\frac{x+2}{3}$, $x\in\br$. We check the equation \eqref{eqattr} for the Cantor set $\mathbf C$: $\tau_0(\mathbf C)$ is a scaled down version of the Cantor set, sitting in $[0,\frac13]$; $\tau_2(\mathbf C)$ is another scaled down version of the Cantor set, sitting in $[\frac23,1]$. Their union is indeed the Cantor set. 

The proof of Theorem \ref{thattr} is based on the Banach fixed point principle: the map $M:H\mapsto \cup_{i=1}^N\tau_i(H)$ is contractive on the set of compact subsets with the Hausdorff metric. Therefore, the attractor is the limit of the iterations $M^n(H)$ for any compact starting set $H$. 

In the case of the middle third Cantor set, let us start with the unit interval $[0,1]$; then $M([0,1])=\tau_0([0,1])\cup\tau_2([0,1])=[0,\frac13]\cup[\frac23,1]$.\\ At the next step $M^2([0,1])=\tau_0(M([0,1]))\cup\tau_2(M([0,1]))=[0,\frac19]\cup[\frac29,\frac13]\cup[\frac23,\frac79]\cup[\frac89,1]$. We see the same usual procedure of constructing the Cantor set.
\end{example}
\begin{remark}
For a general affine iterated system $(\tau_b)_{b\in B}$, the equation \eqref{eqattr} defining the attractor $K=:X_B$ can be rewritten
$$AX_B=\bigcup_{b\in B}(X_B+b).$$
Actually, the attractor has the following representation:
$$X_B=\left\{\sum_{k=1}^\infty A^{-k}b_k\,|\,b_k\in B\right\}.$$
If we assume that there is no overlap between the sets $X_B+b$, $b\in B$, we can construct the function  $\phi_B:=\chi_{X_B}$ and it satisfies
\begin{equation}\label{eqscalattr}
\phi_B(A^{-1}x)=\sum_{b\in B}\phi_B(x-b),\quad(x\in\br^d).
\end{equation}
But that is precisely a form of the scaling equation in multiresolution wavelet theory (see more below)! However the role of the scaling function will not be played here by the attractor of the iterated function system, but by an associated invariant measure. There is a direction of research which exploits equation \eqref{eqscalattr}, which enables one to construct Haar type wavelets on fractal measures.  This was pursued in \cite{DuJo06a,Jonas06} where wavelets were constructed on Cantor sets and Sierpinski gaskets.
\end{remark}

\subsection{Multiresolutions}\label{mra}

Let us recall the key facts of multiresolution wavelet theory. The purpose of multiresolution theory is to construct wavelets, that is orthonormal bases for $L^2(\br^d)$ of the form 
\begin{equation}\label{eqwav}
\left\{|\det A|^{j/2}\psi_i(A^j\cdot-k)\,|\,j\in\bz,k\in\bz^d,i\in\{1,\dots,L\}\right\},
\end{equation}
where $\psi_1,\dots,\psi_L$ are some functions in $L^2(\br^d)$ which will be called {\it wavelets}.

Multiresolutions were introduced by Mallat in \cite{Mal89}. 

\begin{definition}\label{defmul}
A {\it multiresolution} is a sequence of subspaces $(V_n)_{n\in\bz}$ of $L^2(\br^d)$ with the following properties:

\begin{enumerate}
\item $V_n\subset V_{n+1}$, for all $n\in\bz$;
\item $\bigcup_{n\in\bz}V_n$ is dense in $L^2(\br^d)$;
\item $\bigcap_{n\in\bz}V_n=\{0\}$;
\item $f\in V_n$ if and only if $f(A\cdot)\in V_{n+1}$;
\item There exists a function $\varphi\in V_0$ such that
$$\left\{\varphi(\cdot-k)\,|\,k\in\bz^d\right\}$$
is an orthonormal basis for $V_0$.
\end{enumerate}

The function $\varphi$ is called the {\it scaling function}.
\end{definition}

If one has a multiresolution, then wavelets can be constructed by considering the {\it detail space} $W_0:=V_1\ominus V_0$. Here on can find vectors $\psi_1,\dots,\psi_L$ such that their translations $\{\psi_i(\cdot-k)\,|\,k\in\bz^d,i\in\{1,\dots,L\}\}$ form an orthonormal basis for $W_0$. Then, applying the dilation, all the detail spaces $V_{n+1}\ominus V_n$, and the orthonormal basis in (\eqref{eqwav}) is obtained from the 
properties of the multiresolution.
 
\subsection{The scaling equation, the invariance equation}\label{scal}
The scaling function $\varphi$ in Definition \ref{defmul}(v) satisfies an important equation, called the {\it scaling equation}. This is obtained by considering the function $\varphi(A^{-1}\cdot)$ which lies in $V_{-1}\subset V_0$. Since the translates of $\varphi$ form a basis for $V_0$, the scaling equation is obtained:
\begin{equation}\label{eqscaling}
\frac{1}{\sqrt{|\det A |}}\varphi(A^{-1}x)=\sum_{k\in\bz}a_k\varphi(x-k),\quad(x\in\br^d),
\end{equation}
where $(a_k)_{k\in\bz}$ is a sequence of complex numbers.

\medskip
The role of the scaling function in the harmonic analysis of fractal measures is played by an invariant measure. 
\begin{theorem}\label{thinvmea}\cite{Hut81}
Let $(\tau_i)_{i=1}^N$ be a contractive iterated function system on a complete metric space $X$. Let $p_1,\dots, p_N\in[0,1]$ be a list of probabilities.  Then there is a unique probability measure $\mu_p$ on $X$ such that
\begin{equation}\label{eqinvmea}
\mu_p(E)=\sum_{i=1}^Np_i\mu_p(\tau_i^{-1}(E)),\quad\mbox{ for all Borel subsets }E.
\end{equation}
Moreover, the measure $\mu_p$ is supported on the attractor of the iterated function system $(\tau_i)_{i=1}^N$ from Definition \ref{defattr}.
\end{theorem}

In this paper we will work only with equal probabilities $p_i=\frac1N$.
\begin{definition}\label{definvmea}
The unique probability measure $\mu_B$ satisfying 
\begin{equation}\label{eqinvmeab}
\mu_B(E)=\frac{1}{N}\sum_{b\in B}\mu_B(\tau_b^{-1}(E)),\quad\mbox{for all Borel subsets } E,
\end{equation}
is called the {\it invariant measure} associated to the affine iterated function system $(\tau_b)_{b\in B}$.
Equation \eqref{eqinvmeab} is called the {\it invariance equation}.
\end{definition}

\begin{example}
Consider the IFS associated to the middle third Cantor set as in Example \ref{excantor}. The invariant measure $\mu_B$ in this case is the Hausdorff measure corresponding to the Hausdorff dimension $\frac{\ln2}{\ln3}$. In this measure the piece of the Cantor set in $[0,\frac13]$ has measure $\frac12$, the piece of the Cantor set in $[0,\frac19]$ has measure $\frac14$, etc.
\end{example}

\begin{remark}
      It should be stressed that the affine fractals $X$ constructed by IFS-iterations, i.e., from the Cantor-Hutchinson schemes (Theorem \ref{thattr} and Example \ref{excantor}), are non-linear objects. As attractors for IFSs, they are ``chaotic''; see e.g., \cite{Bar06}. This is similarly true for Julia sets from complex dynamics (section 1; i.e., invariant sets that are attractors under iteration of rational functions). They too are non-linear, and they carry no group structure; so there is no available Haar measure.

       We prove that Definition \ref{definvmea} offers a substitute for Haar measure in these fractal settings. One of the conclusions of our work (\cite{DuJo06a, DuJo07b, JoPe94, JoPe98}) is to establish that despite the chaos features of the attractors $X$, there is nonetheless a substitute for Haar measure on these sets $X$. Now the measures from Theorem \ref{thinvmea} form a simplex of measures $\mu_p$ indexed by $p$, where $p = (p_i)_{i=1}^N$ ranges over the probability measures on $\{1, 2,\dots, N\}$. Each of them plays an important role in non-commutative geometry, see \cite{BJO04}.

       In \cite{DuJo07b}, we show that when an IFS is given, then the measures $\mu_p$ on $X$ are induced by infinite product measures corresponding to a choice of probability distribution $(p_i)_{i=1}^N$. The choice of $p_i = 1/N$ in Definition \ref{definvmea} is motivated by the search for spectra for our fractals $X$ in the sense of Definition \ref{defspectral}: A theorem in \cite{DuJo07a} states that for non-uniform weights the Hadamard condition below (Assumption \ref{ass}) is not satisfied, and presumably the measures $\mu_B$ never have spectra. (See also Problem \ref{prob2}.) As a result, the choice of measures in Definition \ref{definvmea} (i.e., uniform weights) serves as a substitute for Haar measure in a non-linear setting. The measures of Brolin \cite{Bro65} supported on Julia sets in the complex plane may also be thought of as substitutes for Haar measure in a fractal setting where there is none. For a fixed Julia set, the associated Brolin measure too has the form  $\mu_p$  for the uniform choice of probability distribution $(p_i)_{i=1}^N$, i.e.,  the choice $p_i = 1/N$ . A Julia set $X$ is constructed from a fixed rational function $R = R(z)$  of degree $N$, and the idea is to get $X = X(R)$ as an IFS attractor for the IFS resulting from a choice of $N$ branches of the inverse of $z\mapsto R(z)$ in the complex plane.
\end{remark}

We can rewrite the invariance equation for continuous compactly supported functions on $\br^d$ as follows:
\begin{equation}\label{eqinvariance}
\int f\,d\mu_B=\frac 1N\sum_{b\in B}\int f\circ\tau_b\,d\mu_B,\quad(f\in C_c(\br^d)).
\end{equation}
or
\begin{equation}\label{eqinv2}
\int f(Ax)\,d\mu_B(x)=\frac1N\sum_{b\in B}\int f(x+b)\,d\mu_B(x),\quad(f\in C_c(\br^d).
\end{equation}
Note that both the scaling equation \eqref{eqscaling} and the invariance equation \eqref{eqinv2} express the dilation of an object in terms of the sum of translated copies of the same object. The resemblance will be more apparent when we take the Fourier transform of both equations.

The Fourier transform of the scaling function is $$\hat\varphi(x)=\int_{\br}\varphi(t)e^{-2\pi i x\cdot t}\,dt.$$
Applying the Fourier transform to the scaling equation \eqref{eqscaling} one obtains:
\begin{proposition}\label{propscalingf}
The following {\it scaling equation} is satisfied:
\begin{equation}\label{eqscalingf}
\hat\varphi(A^Tx)=m_0((A^T)^{-1}x)\hat\varphi((A^T)^{-1}x),\quad(x\in\br^d),
\end{equation}
where $A^T$ is the transpose of the matrix $A$, and 
\begin{equation}\label{eqm0}
m_0(x):=\frac1{|\det A|}\sum_{k\in\bz}a_ke^{2\pi ik\cdot x},\quad(x\in\br^d).
\end{equation}
\end{proposition}
The $\bz^d$-periodic function $m_0$ is called the {\it low-pass filter} in wavelet theory.

For the invariant measure $\mu_B$, the Fourier transform is 
$$\hat\mu_B(x)=\int e^{2\pi i x\cdot t}\,d\mu_B(t)$$
and the invariance equation implies
\begin{proposition}\label{propinvariancef}
The Fourier transform of the measure $\mu_B$ satisfies:
\begin{equation}\label{eqinvariancef}
\hat\mu_B(x)=m_B((A^T)^{-1}x)\hat\mu_B((A^T)^{-1}x),\quad(x\in\br^d),
\end{equation}
where 
\begin{equation}\label{eqmb}
m_B(x)=\frac{1}{N}\sum_{b\in B}e^{2\pi ib\cdot x},\quad(x\in\br^d).
\end{equation}
\end{proposition}
We will call the function $m_B$ the {\it low-pass filter} of the affine IFS $(\tau_b)_{b\in B}$.

We mentioned above that the wavelets are constructed from a multiresolution. The multiresolution is constructed from the scaling function, which in turn is built from the low-pass filter $m_0$. If we iterate formally \eqref{eqscalingf}, and impose the condition $m_0(0)=1$ (from which the name ``low-pass'' comes from), we get the following:
\begin{proposition}\label{propinfprodsc}
 The infinite product formula for $\hat\varphi$ is
\begin{equation}\label{eqinfprodsc}
\hat\varphi(x)=\prod_{k=1}^\infty m_0((A^T)^{-n}x),\quad(x\in\br^d).
\end{equation}
\end{proposition}
The infinite product is uniformly convergent on compact subsets if $m_0$ is assumed to be Lipschitz. 

In the same fashion, the Fourier transform of the invariant measure, $\hat\mu_B$ has an infinite product formula in terms of the low-pass filter $m_B$ (which satisfies $m_B(0)=1$):

\begin{proposition}\label{propinfprod}
The infinite product formula for $\hat\mu_B$ is
\begin{equation}\label{eqinfprodme}
\hat\mu_B(x)=\prod_{k=1}^\infty m_B((A^T)^{-n}x),\quad(x\in\br^d).
\end{equation}
\end{proposition}
The infinite product is uniformly convergent on compact subsets because $m_B$ is a trigonometric polynomial.

\subsection{Orthogonality of the translates of the scaling function; orthogonality of the exponentials}\label{ortho}
For the scaling function, one wants the translates $\varphi(\cdot-k)$, $k\in\bz^d$ to be orthogonal. Taking the Fourier transform and using a periodization, one arrives at the following condition: 

\begin{proposition}\label{proporthosc}
The translates of the scaling function are orthogonal if and only if
\begin{equation}\label{eqorthosc}
h_\varphi(x):=\sum_{k\in\bz}|\hat\varphi(x+k)|^2=1\quad (x\in\br^d).
\end{equation}
\end{proposition}

For the invariant measure $\mu_B$ to be a spectral measure with spectrum $\Lambda$, we need the functions $(e_\lambda)_{\lambda\in\Lambda}$ to be orthogonal in $L^2(\mu_B)$. Writing the Parseval equality for the function $e_{-x}$ in this orthonormal basis one obtains:
\begin{proposition}\label{proporthome}
The set $\Lambda$ is a spectrum for $\mu_B$ if and only if 
\begin{equation}\label{eqorthome}
h_{\Lambda}(x):=\sum_{\lambda\in\Lambda}|\hat\mu_B(x+\lambda)|^2=1\quad(x\in\br^d).
\end{equation}
\end{proposition}
\begin{proof}
The converse is also true, because, taking $x=-\lambda'$ for some $\lambda'\in\Lambda$, the relation \eqref{eqorthome} implies that $e_{\lambda'}$ is orthogonal to all the other $e_\lambda$. And since $e_{-x}$ is in the span of $(e_\lambda)_{\lambda\in\Lambda}$, the Stone-Weierstrass implies that this family spans the entire space $L^2(\mu_B)$.
\end{proof}

\subsection{The quadrature mirror filter condition and Hadamard triples}\label{qmf}
We saw that the scaling function $\varphi$ can be obtained from the low-pass filter $m_0$ by an infinite product formula \eqref{eqinfprodsc}. But we want the translates of the scaling function to be orthogonal. Combining the orthogonality condition \eqref{eqorthosc} with the scaling equation \eqref{eqscalingf}, we see that $m_0$ must satisfy the {\it quadrature mirror filter (QMF)} condition:
\begin{equation}\label{eqqmf}
\sum_{l\in\mathcal L}|m_0((A^T)^{-1}(x+l)|^2=1,\quad(x\in\br^d),
\end{equation}
where $\mathcal L$ is a complete set of representatives for $\bz^d/A^T\bz^d$.

For the moment it is not clear what a QMF condition should be for the fractal measure $\mu_B$. One of the reasons is that we do not have a candidate for the spectrum $\Lambda$ yet. However, we will see that the crucial notion of Hadamard triples introduced by Jorgensen and Pedersen in \cite{JoPe98} can be interpreted as a QMF condition.

\begin{definition}\label{defhada}
Let $A$ be a $d\times d$ integer matrix. Let $B,L$ be two finite subsets of $\bz^d$ of the same cardinality $|B|=|L|=:N$. Then 
$(A,B,L)$ is called a {\it Hadamard triple} if the matrix
\begin{equation}\label{eqhada}
\frac{1}{\sqrt{N}}\left(e^{2\pi i A^{-1}b\cdot l}\right)_{b\in B,l\in L}
\end{equation}
is unitary. (This is called a complex Hadamard matrix.)
\end{definition}

\begin{remark}
The unitarity condition in \eqref{eqhada} may be understood as follows; $d=1$. Suppose $B=\{0,b_1,\dots,b_{N-1}\}$. Then unitarity in \eqref{eqhada} holds for some $L\subset\br$ iff the complex numbers $$\{e^{2\pi i (A^T)^{-1}(l_i-l_k)}\}_{l_i\neq l_k\in L}\subset\bt$$ are roots of the polynomial $1+z^{b_1}+z^{b_2}+\dots+z^{b_{N-1}}$. 

To illustrate the restriction placed on the given pair $(A,B)$, take the example (see Remark \ref{rem4_2}) when $B=\{0,2,3\}$. In that case the associated polynomial equation $1+z^2+z^3=0$ has no solutions with $|z|=1$, and so there is no set $L\subset\br$ which produces a complex Hadamard matrix.

In the case of the middle-third-Cantor example, $\tau_0(x)=x/3$ and $\tau_2(x)=(x+2)/3$; so $A=3$, and $B=\{0,2\}$. Up to a translation in $\br$, the possibilities for the set $L$ are $L=\{0,\frac34(n+\frac12)\}$, for some $n\in\bz$. Since none of these solutions $l=\frac34(n+\frac12)$ are in $\bz$, the unitarity condition \eqref{eqhada} is not satisfied.
\end{remark}
\begin{example}\label{ex4}
We will not use the example of the middle third Cantor set. It was proved in \cite{JoPe98} that this fractal measure does not admit more than $2$ mutually orthogonal exponential functions. However, in \cite{JoPe98}, it was shown that a modification of this does provide an example of a fractal spectral measure. Take $A=4$, $B=\{0,2\}$, so $\tau_0x=x/4$, $\tau_2x=(x+2)/4$. The attractor of this iterated function system is the Cantor set obtained by dividing the unit interval in four equal pieces and keeping the first and the third piece, and iterating this process.

Then one can pick $L=\{0,1\}$ to obtain the Hadamard triple $(A,B,L)$. The unitary matrix in \eqref{eqhada} is $\frac{1}{\sqrt2}\left(\begin{array}{cc}1&1\\ 1&-1\end{array}\right)$. Based on this Jorgensen and Pedersen proved in \cite{JoPe98} that the measure $\mu_B$ has spectrum 
$$\Lambda=\left\{\sum_{k=0}^n4^kl_k\,|\,l_k\in L,n\in\bn\right\}.$$
This was the first example of a (non-atomic) spectral measure which is singular with respect to the Lebesgue measure.
\end{example}

From now on we will make the following:
\begin{assumption}\label{ass}
 There is set $L\subset\bz^d$ of the same cardinality as $B$, $|B|=|L|=:N$ such that $(A,B,L)$ is a Hadamard triple. Moreover we will assume that $0\in B$ and $0\in L$.
\end{assumption}

\begin{proposition}\label{prophada}\cite{LaWa02}
Let $A$ be a $d\times d$ expansive integer matrix, $B,L\subset\bz^d$ with $|B|=|L|=N$. 
Then $(A,B,L)$ is a Hadamard triple if and only if 
\begin{equation}\label{eqqmf2}
\sum_{l\in L}|m_B((A^T)^{-1}(x+l)|^2=1,\quad(x\in\br^d)
\end{equation}
where $m_B$ is the function defined in \eqref{eqmb}.
\end{proposition}
The proof of the proposition requires just a simple calculation.

It is clear that we have now a QMF condition for our fractal setting. 
\begin{definition}\label{defdualifs}
We say that the iterated function system 
$$\sigma_l(x)=(A^T)^{-1}(x+l),\quad(x\in \br^d,l\in L)$$
is {\it dual} to the IFS $(\tau_b)_{b\in B}$ if $(A,B,L)$ is a Hadamard triple.
\end{definition}

With the Hadamard triple we have a first candidate for a spectrum of $\mu_B$:
\begin{equation}\label{eqlambda0}
\Lambda_0:=\left\{\sum_{k=0}^n(A^T)^kl_k\,|\,l_k\in L,n\in\bn\right\}.
\end{equation}

\subsection{Lawton's theorem and transfer operators}\label{lawton}

 In this section, we study a certain operator which is an infinite-dimensional version of a Perron-Frobenius matrix, i.e., a matrix with positive entries. The Perron-Frobenius theorem applies to such matrices $T$, and it states that when $T$ is given, the spectral radius $r = r(T)$ is a distinguished positive eigenvalue of $T$. The corresponding eigenspace contains a positive vector. The peripheral spectrum lies on the circle centered at $0$ in the complex plane with radius $r(T)$. There is an extensive theory involving a certain infinite-dimensional positive operator, which generalizes Perron-Frobenius in a manner which extends the Perron-Frobenius theorem from matrix theory. It is discussed in \cite{BrJo02}, and the reader is referred to this book for details and for additional references.

        Indeed, it turns out that there is an analogous Perron-Frobenius operator in wavelet and fractal theory. We call it ``the transfer operator'', or ``the Perron-Frobenius-Ruelle operator'', but other names may reasonably be associated with it, see e.g., \cite{BrJo02}.

        In the context of wavelets, the operator was identified in a pioneering paper \cite{Law91} by Wayne Lawton. In this wavelet context, the transfer operator $T$ is specified by a chosen filter function $m_0$, and we denote the associated operator $T = R_{m_0}$.

       It is convenient to normalize such that the spectral radius of  $R_{m_0}$ is $1$, and we will do this in our present paper.

       We saw that the scaling relations for wavelets (the equation \eqref{eqscaling} and for fractals (the equations \eqref{eqattr}, \eqref{eqinvmea}, and \eqref{eqinvmeab} ) involve an identity which in a certain sense is of first order. However orthogonality relations involve ``correlations'', and so will be of second order. But it turns out that the quantities that ``measure'' correlations, or inner product numbers for candidates of wavelet functions, or for fractal components, may be understood as eigenfunctions in a certain positive operator, which we call the Ruelle operator, or simply the transfer operator. In its present form, we give it in Definition \ref{deftransfer} below. The functions we are studying in Propositions \ref{proporthosc} and \ref{proporthome} will be eigenfunction for appropriate transfer operators.

       The best orthogonality relations in the analysis of wavelets and of fractals will come about from testing the leading eigenspace of the transfer operator; i.e., the eigenspace of the top eigenvalue in the spectrum of our transfer operator, as we make precise in Propositions \ref{prophphi} and \ref{thdutjor1} below.\medskip

Suppose $m_0$ is a Lipschitz low-pass filter ($m_0(0)=1$) which satisfies the QMF condition \eqref{eqqmf}. Then one can construct the scaling function $\varphi$ via the infinite product formula \eqref{eqinfprodsc}. Moreover, $\hat\varphi$ defined by this infinite product is a function $L^2(\br^d)$ (\cite{Mal89,Jor06, HeWe96}), and if $m_0$ is a trigonometric polynomial, i.e. only a finite number of the coefficients $a_k$ are nonzero, the infinite product will give us a compactly supported $L^2(\br^d)$-function (\cite[Theorem 3.13]{HeWe96}). However the QMF condition is only necessary, it is not sufficient to guarantee the orthogonality of the translates of the scaling function. More is needed. The extra ingredient was provided by Lawton in \cite{Law91}.

\begin{theorem}\label{thlawton}\textup{(Lawton \cite{Law91})}
Let $m_0$ be a Lipschitz low-pass filter with $m_0(0)=1$ and satisfying the QMF condition \eqref{eqqmf}. Then the scaling function $\varphi$ defined by \eqref{eqinfprodsc} is an orthonormal scaling function (i.e., its translates are orthonormal) if and only if the only continuous $\bz^d$-periodic functions $h$ that satisfy
\begin{equation}\label{eqlawton}
\sum_{l\in\mathcal L}|m_0((A^T)^{-1}(x+l)|^2h((A^T)^{-1}(x+l))=h(x),\quad(x\in\br^d)
\end{equation}
are the constants. (Recall that $\mathcal L$ is a complete set of representatives for $\bz^d/A^T\bz^d$.)
\end{theorem}

\begin{proof}
The main observation needed here is that the function $h_\varphi$ of \eqref{eqorthosc} satisfies the equation \eqref{eqlawton}. It can be proved continuous, since $m_0$ is Lipschitz. So if the only solutions to \eqref{eqlawton} are the constants, then $h_\varphi=1$ because $h_\varphi(0)=1$. The proof of the converse statement requires a minimality property of $h_\varphi$: it is the smallest non-negative continuous solution to \eqref{eqlawton} which satisfies $h_\varphi(0)=1$ (see \cite[Section 5.1]{BrJo02}, \cite{Dut04}). If it is equal to $1$, then there can be no other such functions.
\end{proof}

Lawton's theorem brings a new character into play: the transfer operator.
\begin{definition}\label{deftransfer}
The operator defined on $\bz^d$-periodic functions $f$ on $\br^d$ by
\begin{equation}\label{eqtransfer}
(R_{m_0}f)(x)=\sum_{l\in \mathcal L}|m_0((A^T)^{-1}(x+l))|^2f((A^T)^{-1}(x+l)),\quad(x\in\br^d)
\end{equation}
is called the {\it transfer operator}.

In the fractal setting, one uses the function $m_B$ for the ``weight'' in place of $m_0$, and the dual IFS $(\sigma_l)_{l\in L}$ to replace the inverse branches $x\mapsto (A^T)^{-1}(x+l)$. Thus the transfer operator is defined by
\begin{equation}\label{eqtransfer2}
(R_{m_B}f)(x)=\sum_{l\in L}|m_B(\sigma_l(x))|^2f(\sigma_l(x)),\quad(x\in\br^d).
\end{equation}
\end{definition}

\begin{proposition}\label{prophphi}The following affirmations hold:
\begin{enumerate}
\item In the wavelet context, with $h_\varphi$ defined in \eqref{eqorthosc}, one has $R_{m_0}h_\varphi=h_\varphi$.
\item In the fractal context, if $\Lambda_0$ is the candidate for a spectrum, defined in \eqref{eqlambda0}, and $h_{\Lambda_0}$ is the function defined in \eqref{eqorthome}, one has $R_{m_B}h_{\Lambda_0}=h_{\Lambda_0}$.
\end{enumerate}

\end{proposition}

\begin{proof}
Since we have already talked about the wavelet context (i), we only have to see what happens in (ii). The ideas are the same as in the wavelet context. One has to perform a computation. The important property needed here is that $\Lambda_0=\cup_{l\in L}(l+A^T\Lambda_0)$, and the union is disjoint (a consequence of the Hadamard property \eqref{eqhada}). For details we refer to \cite{DuJo06b}.
\end{proof}

\begin{theorem}\label{thdutjor1}
The set $\Lambda_0$ in \eqref{eqlambda0} is a spectrum for $\mu_B$ if and only if the only continuous eigenfunctions of the transfer operator $R_{m_B}$ with eigenvalue $1$ are the constants.
\end{theorem}

\begin{proof}
The main ideas are the same as in the proof of Theorem \ref{thlawton}. While we did not formulate this theorem in \cite{DuJo06b}, all the required details are contained there.
\end{proof}

\subsection{Cohen's orthogonality condition and the \L aba-Wang theorem}\label{cohen}
A different necessary and sufficient condition for the orthogonality of the translates of the scaling function is captured in {\it Cohen's condition}. Before stating it, we need a definition.

\begin{definition}\label{defcycle}
Let $m_0$ be a low-pass filter satisfying the quadrature mirror filter condition \eqref{eqqmf}. A finite set $C:=\{x_1,\dots,x_{p-1}\}$ in $[0,1)^d$ is called a {\it cycle} if $Ax_i\equiv x_{i+1}\mod\bz^d$ for $i\in\{1,\dots,p-1\}$ and $Ax_p\equiv x_0\mod\bz^d$. The cycle $C$ is called an {\it $m_0$-cycle} if $|m_0(x_i)|=1$ for all $i\in\{1,\dots,p\}$.
\end{definition}

\begin{theorem}\label{thcohen}\textup{(Cohen \cite{Coh90,Coh90a})} 
Let $m_0$ be a Lipschitz low-pass filter satisfying the QMF condition \eqref{eqqmf}. Assume in addition that $m_0$ has finitely many zeros. Let $\varphi$ the scaling function defined by the infinite product formula \eqref{eqinfprodsc}. Then the translates of $\varphi$ are orthogonal if and only if the only $m_0$-cycle is the trivial one $\{0\}$.
\end{theorem}

\begin{proof}
The idea of the proof is to use Lawton's theorem (theorem \ref{thlawton}) and consider a continuous function $h\geq0$ such that $R_{m_0}h=h$. One can arrange $h$ so that is has some zeros. If $x_0$ is a zero for $h$, then $|m_0((A^T)^{-1}(x+l))|^2h((A^T)^{-1}(x+l))=0$ for all $l\in\mathcal L$. Using the QMF condition \eqref{eqqmf}, at least one of the terms $|m_0((A^T)^{-1}(x+l))|^2$ is non-zero. So one obtains another zero of $h$, $x_1$ such that $A^Tx_1=x_0\mod\bz^d$. The process can not continue indefinitely, and the points must end in a cycle. For details see \cite{Coh90,Coh90a,DuJo06b}.
\end{proof}

In the fractal context one looks for conditions when $\Lambda_0$ defined in \eqref{eqlambda0} is a spectrum for the measure $\mu_B$. We will need a definition analogous to Definition \ref{defcycle} for iterated function systems:
\begin{definition}\label{defcycleifs}
A set $C:=\{x_1,\dots,x_p\}$ is called a {\it cycle} for the IFS $(\sigma_l)_{l\in L}$ from Definition \ref{defdualifs}, if $\tau_{l_i}x_i=x_{i+1}$ for $i\in\{1,\dots,p-1\}$, $\tau_{l_p}x_p=x_1$ for some $l_1,\dots,l_p\in L$. Let $m_B$ be the function defined in \eqref{eqmb}. The cycle $C$ is called an {\it $m_B$-cycle} if $|m_B(x_i)|=1$ for all $i\in\{1,\dots,p\}$.
\end{definition}

\begin{theorem}\label{thlabawang}\textup{(\L aba-Wang \cite{LaWa02})}
Under Assumption \ref{ass}, suppose that $m_B$ has only finitely many zeros in the attractor $X_L$ of the IFS $(\sigma_l)_{l\in L}$. Then the set $\Lambda_0$ defined in \eqref{eqlambda0} is a spectrum for the invariant measure $\mu_B$ if and only if the only $m_B$-cycle is the trivial one $\{0\}$.
\end{theorem}

\begin{remark} \L aba and Wang proved this theorem for dimension $d=1$ , but their argument works also in higher dimensions, see \cite{DuJo06b}. The proof follows the same ideas as in the wavelet case.
\end{remark}

\subsection{Super-wavelets and the completion of $\Lambda_0$}\label{super}

What happens in wavelet theory if the translates of the scaling function are not orthogonal, so Lawton's condition in Theorem \ref{thlawton} and Cohen's condition in Theorem \ref{thcohen} fail? One still gets an interesting property of the associated wavelet $\psi$ in $L^2(\br^d)$ does not generate an orthonormal basis but a {\it Parseval frame} in \eqref{eqwav} (see \cite{Law91,Dau92,BrJo02}).

\begin{definition}\label{defparseval}
A family of vectors $(e_i)_{i\in I}$ in a Hilbert space $H$ is called a Parseval frame if 
$$\|f\|^2=\sum_{i\in I}|\ip{f}{e_i}|^2,\quad(f\in H).$$
\end{definition}

The discovery that that the MRA wavelets are Parseval frames even in the case when they might not be orthogonal is usually credited to the pioneering paper of Wayne Lawton \cite{Law91}. In the book \cite{BrJo02} this phenomenon is identified with a natural class of representations of the Cuntz algebras, one that in turn is motivated by signal processing in communication engineering. The authors of \cite{BrJo02} prove that the Parseval property follows from the axioms that defining these Cuntz algebra representations. In a more general context, for the generalized multiresolutions (GMRAs) the idea was extended and worked out in \cite{BJMP05}.

Cohen's condition implies that there are some non-trivial $m_0$-cycles. It was shown in \cite{Dut04,BDP05} (for dimension $d=1$) that these $m_0$-cycles can be used to construct a ``super-wavelet'': that is a vector $\tilde\psi$ in a ``super-space'' of the form $L^2(\br)\oplus\dots\oplus L^2(\br)$ such that bay applying certain dilations and translations to $\tilde\psi$ (as in \eqref{eqwav}) one obtains an orthonormal basis. Moreover, the classical wavelet $\psi$ in $L^2(\br)$ and its Parseval frame is obtained by projecting the super-wavelet basis onto the first component. Also, the wavelet is obtained by a multiresolution technique, applied in the super-space, but with the same filter $m_0$.

\begin{theorem}\label{thbdp}\textup{(Bildea-Dutkay-Picioroaga, \cite{Dut04,BDP05})}
When Cohen's condition in Theorem \ref{thcohen} is not satisfied, then $m_0$ generates a multiresolution orthogonal super-wavelet and an orthogonal scaling function for a space of the form $L^2(\br)\oplus\dots\oplus L^2(\br)$. Moreover the projection of these onto the first component gives the classical scaling function and Parseval wavelet in $L^2(\br)$.
\end{theorem}

\begin{remark}
The super-wavelet construction in $L^2(\br)\oplus\dots\oplus L^2(\br)$ requires the two unitary operators: the dilation and translation operators. They are defined in \cite{Dut04,BDP05}. We will not include the precise definitions here but we will give an example which contains most of the important ideas.
\end{remark}

\begin{example}\label{exstretched}\textup{(The stretched Haar wavelet)}
Let $d=1$, $A=2$ and $m_0(x)=\frac{1}{2}(1+e^{2\pi i 3x})$. This is a low-pass filter, $m_0(0)=1$ and it satisfies the QMF condition \eqref{eqqmf}. The infinite product formula yields 
$$\varphi=\frac13\chi_{[0,3)}.$$ This satisfies the scaling equation \eqref{eqscaling}, but the translates of it are clearly not orthogonal. The associated wavelet is $$\psi=\frac{1}{3}\chi_{[0,\frac32)}-\frac13\chi_{[\frac32,3)}.$$
The family in \eqref{eqwav} is not an orthonormal basis, but it is a Parseval frame.

Note that there is a non-trivial $m_0$-cycle, namely $\{\frac13,\frac23\}$. On the Hilbert space $L^2(\br)\oplus L^2(\br)\oplus L^2(\br)$ define the dilation and translation operators 
$$U_3(f_1,f_2,f_3)(x)=(Uf_1,Uf_3,Uf_2),\quad T_3(f_1,f_2,f_3)=(Tf_1,e^{2\pi i\frac13}Tf_2,e^{2\pi i\frac23}Tf_3),$$
where $Uf(x)=\frac{1}{\sqrt2}f(x/2)$, $Tf(x)=f(x-1)$ are the usual dilation and translation operators in $L^2(\br)$.

Define the super-scaling function and super-wavelet by
$$\varphi_3=(\varphi,\varphi,\varphi),\quad \psi_3=(\psi,\psi,\psi).$$

The results in \cite{Dut04,BDP05} prove that
\begin{enumerate}
\item
 The super-scaling function $\varphi_3$ satisfies the scaling equation with filter $m_0$:
 $$U_3\varphi_3=\frac{1}{\sqrt2}(\varphi_3+T_3^3\varphi_3).$$
\item 
The translates of $\varphi_3$ are orthogonal:
$$\ip{T_3^k\varphi_3}{T_3^{k'}\varphi_3}=\delta_{k,k'},\quad (k,k'\in\bz).$$
\item The super-wavelet $\psi_3$ generates the orthonormal basis:
$$\{U_3^jT_3^k\psi_3\,|\,j,k\in\bz\}$$
is an orthonormal basis for $L^2(\br)\oplus L^2(\br)\oplus L^2(\br)$.
\end{enumerate}
It is clear that results of the theory in $L^2(\br)$ for this example are recuperated by a projection onto the first component.

\begin{remark}
The fact that the components of $\varphi_3$ and $\psi_3$ are the same in this case is just a coincidence. In general, the components of the super-scaling function are defined by an infinite product formula similar to \eqref{eqinfprodsc}, but involving the cycle. Then the super-wavelet is defined from the super-scaling function, just as in the classical multiresolution theory, but using the dilation and translation of the super-space. We refer again to \cite{Dut04,BDP05} for details.
\end{remark}
\end{example}

Let us look now at the fractal measure $\mu_B$, and see what happens if the condition in the \L aba-Wang theorem is not satisfied, that is there are some non-trivial $m_B$-cycles. In this case the exponentials $e^{2\pi i\lambda x}$ with $\lambda\in \Lambda_0$ do form an orthonormal set, but the set is incomplete. It turns out that, even in this case, with some extra assumptions on $m_B$, $\mu_B$ is a spectral measure, and each $m_B$-cycle will contribute to a subset of the spectrum. Putting together these contributions one obtains the entire spectrum for $\mu_B$, and $\Lambda_0$ is just the contribution of the trivial $m_B$-cycle $\{0\}$.

\begin{theorem}\label{thdujo}\textup{(Dutkay-Jorgensen \cite{DuJo06b})}
Suppose Assumption \ref{ass} holds and $m_B$ has only finitely many zeros on the attractor of the IFS $(\sigma_l)_{l\in L}$. Let $\Lambda$ be the smallest set that contains $-C$ for all $m_B$-cycles $C$, and which satisfies the invariance property
$$A^T\Lambda+L\subset\Lambda.$$
Then $\Lambda$ is a spectrum for the measure $\mu_B$.
\end{theorem}

\begin{example}\label{ex4_2}
Consider the affine IFS in Example \ref{ex4}. $d=1$, $A=4$, $B=\{0,2\}$ so $\tau_0(x)=x/4$, $\tau_2(x)=(x+2)/4$. We saw that we can pick $L=\{0,1\}$ and this will yield the spectrum $\Lambda=\{\sum_{k=0}^n4^kl_k\,|\, l_k\in\{0,1\}\}$ for the fractal measure $\mu_B$. The dual IFS is $\sigma_0(x)=x/4$, $\sigma_1(x)=(x+1)/4$. The function $m_B(x)=\frac12(1+e^{2\pi i 2x})$ will have only the trivial $m_B$-cycle.

But we can also pick $L=\{0,3\}$. The Assumption \ref{ass} is verified. The dual IFS is now $\sigma_0(x)=x/4$, $\sigma_3(x)=(x+3)/4$. There are two $m_B$-cycles now. One is the trivial one $\{0\}$ (with $\sigma_00=0$). The other is $\{1\}$ with $\sigma_31=1$. The spectrum $\Lambda_3$ given by Theorem \ref{thdujo} is
$$\Lambda_3:=\left\{\sum_{k=0}^n4^kl_k\,|\,l_k\in\{0,3\}\right\}\cup\left\{-1+\sum_{k=0}^n4^kl_k\,|\,l_k\in\{0,-3\}\right\}.$$

Actually, one can take any $L=\{0,p\}$ with $p$ odd, and with Theorem \ref{thdujo} will yield a different spectrum for $\mu_B$. And there does not seem to be any direct relation between these spectra for different values of $p$.
\end{example}

\section{Open problems}\label{open}
\begin{problem}\label{prob1}\textup{(\cite{DuJo07b})}
Under Assumption \ref{ass} prove that $\mu_B$ is a spectral measure.
\end{problem}

Theorem \ref{thdujo} shows that this statement is true in dimension $d=1$. In \cite{DuJo07b} we proved that this is true also for higher dimension if a certain extra technical condition is imposed on $(A,B,L)$. This condition allows infinitely many zeros for $m_B$ in the attractor of the IFS $(\sigma_l)_{l\in L}$, so the results of \cite{DuJo07b} are more general than Theorem \ref{thdujo}. However we were not able to remove this extra condition, but we strongly believe that the result remains true without it. Another possibility is that the condition is automatically satisfied under Assumption \ref{ass}.

\begin{problem}\label{prob2}\textup{(\cite{LaWa02})}
Let $\mu_{\mathcal B,p}$ be the invariant probability measure associated to the IFS $\tau_b(x)=\rho(x+\beta)$, $\beta\in \mathcal B$, and probabilities $(p_i)_{i=1}^N$, with $|\rho|<1$ and $\mathcal B\subset \br$, $|B|=N$, $p_i>0$. 

Suppose $\mu_{\mathcal B,p}$ is a spectral measure. Then

\begin{enumerate}
\item $\rho=\frac1P$ for some $P\in\bz$.
\item $p_1=\dots=p_N=\frac1N$.
\item Suppose $0\in \mathcal B$. Then $\mathcal B=\alpha B$ for some $\alpha\neq 0$ and some $B\subset\bz$ such that there exists $L$ that makes $(P,B,L)$ a Hadamard triple.
\end{enumerate}
\end{problem}

We obtained some results in this direction in \cite{DuJo07a}, but we are still far from settling this problem.

\begin{problem}\label{prob3}
Under Assumption \ref{ass}, prove that 
$$\mu_B(\tau_b(X_B)\cap\tau_{b'}(X_B))=0,\quad(b,b'\in B,b\neq b')$$
\end{problem}

Thus we are interested if the affine IFS has (measurable) overlap. \medskip

The next Problem discusses how IFS-Hutchinson measures $\mu$ depend on the geometry in an ambient $\br^d$, i.e, the initialization of the IFS recursive scheme for $\mu$.

        The related discussion of Corollary \ref{cor1_1} in the Introduction involves analyticity, so the ambient space is the complex plane and begins with the open set, the disk $\{z\, |\, |z| < 1\}$. For this case, the selfsimilarity pattern that initializes our IFS is identified on the boundary circle. While the IFS-Hutchinson measures typically have a definite fractal dimension, there are many ways of arriving at them from some choice of transformations in an ambient $\br^d$.

       In the next section, we discuss in detail the interplay between open sets in 
       $\br^d$ with selfsimilarity patterns, and an associated class of IFS-Hutchinson measures arising from an iteration-limit applied to an initial geometric pattern in $\br^d$. Hence we may ask for the property in Definition \ref{defspectral} (that the measure is spectral) to hold at the first step, i.e., for the restriction of Lebesgue measure to a suitably chosen open set $\Omega$ in $\br^d$. We are asking about how this property predicts the IFS-measures $\mu$ arising from a fractal iteration scheme: Specifically, is an IFS-measure $\mu$ spectral if its initialized counterpart is? Here ``the initialized counterpart'' refers to the data in (ii) below; and it will be taken up in further detail in the last section (Theorem \ref{ThmFJP} in our paper).

        We now make this theme precise in the following:
\begin{problem}\label{prob4}
 Consider an affine system given by the pair $(A, B )$ in $\br^d$ with the usual conditions on the $d$ by $d$ matrix $A$ and the subset $B$:\begin{enumerate}
\item[(a)]
The matrix $A$ has integral entries and is expansive; and
\item[(b)] 0 is in the set $B$, and $B$ is contained in $\bz^d$.
\end{enumerate}
Then the following three affirmations are equivalent:
\begin{enumerate}
\item There is a subset $L$  contained in $\bz^d$ such that $|L| = |B| (= N)$ and $(A, B , L)$ is a Hadamard triple.
\item If the system $(A, B)$ is initialized with an open subset $\Omega$ in $\br^d$, then the restriction of Lebesgue measure to $\Omega$ is spectral. (Is the initial measure spectral?)
\item The IFS Hutchinson-measure $\mu$ constructed from the initial system $(A, B )$ in $\br^d$ by the recursive recipe in Definitions \ref{defaff} and \ref{definvmea} is spectral. (Is the final measure spectral?)
\end{enumerate}
 \end{problem}

\begin{problem}\label{prob5}
Characterize the polynomials $z \mapsto R(z)$ for which there is an isometric isomorphism $W$ of $L^2(J(R), \mu_R)$  with the Hilbert space $K^2(R, P)$ for some finite set $P (= P(R))$ of distinct polynomials each of degree less than the degree of $R$, and with $W$ taking the form $F\mapsto (F_p)_{p \in P}$, where $F(z) = \sum_{p \in P} p(z) F_p(R(z))$.
\end{problem}

\begin{remark}\label{rem3dz}
Specifically, the requirement on the isometry $L^2(\mu)\ni f\stackrel{W}{\rightarrow}F_f\in K^2(R,P)$ is as follows: $W$ is defined on $L^2(\mu)$ and maps onto the subspace $K^2(R,P)$ in $H^2$.

Moreover
$$\|f\|_{L^2(\mu)}^2=\|F_f\|_{H^2}^2=\sum_{p\in P}\|F_p\|_{H^2}^2.$$
\end{remark}

%\appendix
\section{Appendix, with some of Jorgensen's recollections of conversations
with Irving~Segal, Marshall~Stone, Bent~Fuglede, and David~Shale.}\label{app}

While this paper emphasizes orthogonality, stressing fractal measures, many of the problems have a history beginning with the case when the measure $\mu$ under consideration is a restriction of Lebesgue measure on $\br^d$. The original problem was motivated by von Neumann's desire to use his Spectral Theorem on partial differential operators (PDOs), much like the Fourier methods had been used boundary value problems in ordinary differential equations.

   What follows is an oral history, tracing back some key influences in the subject, step-by-step, three or four mathematical generations. The reader is cautioned that it is a subjective account. And it should be added that analysis of fractal measures derived inspiration from a diverse variety of sources: potential theory, geometric measure theory, the study of singularities of solutions to partial differential equations (PDEs), and dynamics. Here the focus is on operator theory and harmonic analysis, subjects which in turn derived their inspiration from much the same sources, and from quantum physics. A central character in our account below is {\it projection valued measures} as envisioned by John von Neumann \cite{vNeu32b}and Marshall Stone \cite{Sto32b}.

    In what follows we trace influences from operator theory and mathematical physics, with projection valued measures playing a key role. Much of it was passed onto us from generations back.

We begin the history of what we now call the Fuglede problem \cite{Fug74} with conversations the second named author had with Irving Segal, Marshall Stone, Bent Fuglede, and David Shale. Since this never found its way into print, in this appendix, we draw up connections to early parts of operator theory.

Starting in the early fifties, only a few years after the end of WWII, and
still in the shadow of the war, mathematical physicists returned to the
center of attraction around Niels Bohr in Copenhagen. Both Niels Bohr, and
his mathematician brother Harald were still active; and still inspired
members of the next generation. This is at the same time the generation of
mathematicians who inspired me, along with a large number of other
researchers in functional analysis, in operator theory, in partial
differential equations, in harmonic analysis, and in mathematical physics.

Included in the flow of scientists who visited Copenhagen around the time
was some leading analysts from the University of Chicago, and elsewhere in
the US, Marshal Stone, Irving Segal, Richard Kadison, George Mackey, and Arthur Wightman to
mention only a few. What this group of analysts has in common is that the
members were all much inspired by John von Neumann's view on mathematical
physics and operator theory. From conversations with Segal, I sense that von
Neumann was viewed as a sort of Demi-God.

The forties and the fifties was the Golden Age of functional analysis, the
period when the spectral theorem, and the spectral representation theorem,
were being applied to problems in harmonic analysis, and in quantum field
theory. Here are some examples: (1) Irving Segal's Plancherel theorem for
unitary representations of locally compact unimodular groups \cite{Seg50};
(2) the Gelfand--Naimark--Segal construction of representations of $C^{\ast }
$-algebras from quantum mechanical states \cite{Nai56,Nai59}; (3) Stone's
theorem for unitary one-parameter groups; and its generalization \cite%
{Sto32b}, (4) the Stone--Naimark--Ambrose--Godement theorem for
representations of locally compact abelian groups; (5) the rigorous
formulation of the Stone--von Neumann uniqueness theorem \cite%
{vNeu32b,vNeu68}; and (6) its generalization into what became the Mackey
machine for the study of unitary representations of semi-direct products of
continuous groups, including (7) the imprimitivity theorem \cite{Mac52,Mac53}
for induced representations of unitary representations of locally compact
groups. See also \cite{SeGo65,GaWi54,Hor55,Sha62}.

In this context, we find much work involving unbounded operators in Hilbert
space, typically non-commuting operators, for example axiomatic formulations
of the momentum and position operators, and in a context where the dense
domain of the operators must be taken very seriously. The simplest such
operators are the momentum and position operators $Q$ and $P$ that form toy
examples of quantum fields; i.e., the operators of Heisenberg's uncertainty
inequality \cite{GaWi54}.

While the spectral theorem had already played a major role in mathematical
physics, at the same time Lars H\"{o}rmander \cite{Hor55} and Lars G\aa rding \cite{GaWi54} in Sweden
were using functional analytic methods in the study of elliptic problems in
PDE theory.

In his first visit to Copenhagen in the early fifties, Irving Segal brought
with him one of his graduate students, David Shale, who was supposed to have
written a thesis on commuting selfadjoint extensions of partial differential
operators, and apparently achieved some initial results; never published in
any form.

\textbf{A baby example:} To understand the derivative operator $d/dx$ in a
finite interval, we know that von Neumann's deficiency indices help. Von
Neumann's deficiency indices is a pair of numbers $(n,m)$ which serve as
obstructions for formally hermitian operators with dense domain in Hilbert
space to be selfadjoint.

For a single formally hermitian operator in Hilbert space $\mathcal{H}$, it
is known that selfadjoint extensions exist if and only if $n=m$. The size of 
$n$ ($=m$) is a measure of the cardinality of the set of selfadjoint
extensions: The different selfadjoint extensions are parameterized by
partial isometries between two subspaces in $\mathcal{H}$ each of dimension $%
n$.

The property of \emph{selfadjointness} is essential, as the spectral theorem
doesn't apply to formally hermitian operators. If the usual derivative
operator $d/dx$ is viewed as an operator in $L^{2}\left( 0,1\right) $ with
dense domain $\mathcal{D}$ consisting of differentiable functions vanishing
at the end-points, then one checks that it is formally skew-hermitian, and
that its deficiency indices are $(1,1)$. As a result, we see that the
derivative operator with zero-endpoint conditions has a one-parameter family
of selfadjoint extensions; the family being indexed by the circle. In fact,
the family has a simple geometric realization as follows: Thinking of
functions in $L^{2}\left( 0,1\right) $ as wave functions, there is then
clearly only one degree of freedom in choosing selfadjoint extensions. As we
translate the wave function to the right, and hit the boundary point, a free
phase must be assigned. This amounts to the `closing up' the wave functions
at the endpoints of the unit interval. What goes out at $x=1$ must return at 
$x=0$, possibly with a fixed phase correction. With this in mind, it was
natural to ask for an analogue of von Neumann's deficiency indices for
several operators. In fact this is a dream that was never realized, but one
which still serves to inspire me.

In the context of $\mathbb{R}^d$, i.e., higher dimensions, it is very
natural to ask the similar extension question for bounded open domains $\Omega$
in $\mathbb{R}^d$. And it was clear what to expect.

Going to $d$ dimensions: If functions in $L^{2}\left( \Omega\right) $ are
translated locally in the $d$ different coordinate directions, we will
expect that the issue of selfadjoint extension operators should be related
to the matching of phases on the boundary of $\Omega$, and therefore related to
the \emph{tiling} of $\mathbb{R}^{d}$ by translations of $\Omega$; i.e., with
translations of $\Omega$ which cover $\mathbb{R}^{d}$, and which do not overlap
on sets of positive Lebesgue measure. The global motion by continuous
translation in the $d$ coordinate directions will be determined uniquely by
the spectral theorem if we can find commuting selfadjoint extensions of the $%
d$ partial derivative operators 
\begin{equation*}
i\frac{\partial }{\partial x_{j}},\qquad j=1,\dots ,d,
\end{equation*}%
defined on the dense domain $\mathcal{D}$ of differentiable functions on $\Omega$
which vanish on the boundary. We can take $\mathcal{D}=C_{c}^{\infty }\left(
\Omega\right) $. These $d$ operators are commuting and formally hermitian, but
not selfadjoint. In fact when $d>1$, each of the operators has deficiency
indices $(\infty ,\infty )$. So in each of the $d$ coordinate directions, $%
i\,\partial /\partial x_{j}|_{\mathcal{D}}$ has an infinite variety of
selfadjoint extensions. But experimentation with examples shows that `most'
choices of $\Omega$ will yield non-commuting selfadjoint extensions. Each
operator individually does have selfadjoint extensions, and the question is
if they can be chosen to be mutually commuting. By this we mean that the
corresponding projection-valued spectral measures commute.

The spectral representation for every selfadjoint extension $H_{j}\supset
i\,\partial /\partial x_{j}|_{\mathcal{D}}$ has the form $H_{j}=\int_{%
\mathbb{R}}\lambda \,E_{j}\left( d\lambda \right) $, $j=1,\dots ,d$, where $%
E_{j}\colon \mathcal{B}\left( \mathbb{R}\right) 
\rightarrow \operatorname{Projections}\left( L^{2}\left( \Omega\right) \right) 
$ denotes the Borel subsets
of $\mathbb{R}$. We say that a family of $d$ selfadjoint extensions $%
H_{1},\dots ,H_{d}$ of the respective $i\,\partial /\partial x_{j}$
operators is commuting if $E_{j}\left( A_{j}\right) E_{k}\left( A_{k}\right)
=E_{k}\left( A_{k}\right) E_{j}\left( A_{j}\right) $ for all $A_{j},A_{k}\in 
\mathcal{B}\left( \mathbb{R}\right) $, and $j\neq k$.

When commuting extensions exist, we form the product measure%
\begin{equation*}
E=E_{1}\times \dots \times E_{d}
\end{equation*}%
on $\mathcal{B}\left( \mathbb{R}^{d}\right) $, and set%
\begin{equation*}
U\left( t\right) =\int_{\mathbb{R}^{d}}e^{i\lambda \cdot t}\,E\left(
d\lambda \right) ,
\end{equation*}%
where $t,\lambda \in \mathbb{R}^{d}$ and $\lambda \cdot
t=\sum_{j=1}^{d}\lambda _{j}t_{j}$. Then clearly 
\begin{equation*}
U\left( t\right) U\left( t^{\prime }\right) =U\left( t+t^{\prime }\right)
,\qquad t,t^{\prime }\in \mathbb{R}^{d},
\end{equation*}%
i.e., $U$ is a unitary representation of $\mathbb{R}^{d}$ acting on $%
\mathcal{H}=L^{2}\left( \Omega\right) $.

\begin{theorem}
\label{ThmFJP}\textup{(Fuglede--Jorgensen--Pedersen \cite{Fug74,Jor82,Ped87})%
} Let $\Omega$ be a non-empty connected open and bounded set in $\mathbb{R}^{d}$, and
suppose that the $d$ operators%
\begin{equation*}
i\frac{\partial }{\partial x_{j}}\bigg|_{C_{c}^{\infty }\left( \Omega\right)
},\qquad j=1,\dots ,d,
\end{equation*}%
have commuting extensions in $L^{2}\left( \Omega\right) $. Then there is a subset 
$S\subset \mathbb{R}^{d}$ such that%
\begin{equation*}
\left\{ \,e^{is\cdot x}\mid s\in S\,\right\} 
\end{equation*}%
is an orthogonal basis in $L^{2}\left( \Omega\right) $.
\end{theorem}

\begin{remark}\label{remn}
It might be natural to expect that an open spectral set $\Omega$ will have its connected components spectral, or at least have features predicted by the spectrum of the bigger set. This is not so as the following example (due to Steen Pedersen) shows. Details below!

   If the set $\Omega$ in Theorem \ref{ThmFJP} is not assumed connected, the conclusion in Theorem \ref{ThmFJP} would be false. If a spectral set $\Omega$ is disconnected, then properties of the connected components are not immediately discerned from knowing that there is a spectrum for $\Omega$.

     An example showing this can be constructed by taking $\Omega = \Omega(p)$ to the following set obtained from a unit square, dividing it into two triangles along the main diagonal, followed by a translation of the triangle under the diagonal.

   Details: Start with the following two open triangles making up a fixed unit-square, divided along the main diagonal. Now make a translation of the triangle under the diagonal by a non-zero integer amount $p$ in the $x$-direction; i.e., by the vector $(p,0)$, leaving the upper triangle alone .

Further, let $\Omega(p)$ be the union of the resulting two $p$-separated triangles. Hence the two connected components in  $\Omega(p)$ will be two triangles; the second obtained from the first by a mirror image and a translation. Neither of these two disjoint open triangles is spectral; see \cite{Fug74}. Nonetheless, as spectrum for $\Omega(p)$ we may take the unit-lattice $\bz^2$. To see this, one may use a simple translation argument in  $L^2(\Omega(p))$, coupled with the fact that $e_\lambda(p) = 1$   for all $\lambda \in\bz^2$.

There are two interesting open connected sets in the plane that are known \cite{Fug74} not to be spectral. They are the open disk and the triangle. Of those two non-spectral sets, the triangles may serve as building blocks for spectral sets. Not the disks!

\end{remark}

\begin{remark}\label{rem4_2}
It is enough to assume that $\Omega\subset \mathbb{R}^{d}$ is open and has
positive finite Lebesgue measure.

If the set $\Omega$ is the theorem is disconnected, then it is not necessarily true that there is a single set $S$, such that $(\Omega,S)$ is a spectral pair.

To see this, take $d=1$ and $\Omega=(0,1)\cup(2,4)$, i.e., two open intervals as specified. The reason that there cannot be a set $S$ such that $\{e_s\,|\,s\in S\}$ is an ONB in $L^2(\Omega)$ is that the polynomial $z^4-z^2+z-1$ does not have enough zeros to make $(\Omega,S)$ a spectral pair.
\end{remark}

\begin{proof}
Fuglede proved, under more stringent assumptions, that the spectrum of $E$
is atomic, i.e., as a measure, $E$ is supported on a discrete subset $%
S\subset \mathbb{R}^{d}$. Moreover, each $E\left( \left\{ s\right\} \right)
L^{2}\left( \Omega\right) $ for $s\in S$ is one-dimensional, and $E\left( \left\{
s\right\} \right) L^{2}\left( \Omega\right) =\mathbb{C}\,e^{is\cdot x}$, where $%
e_{s}\left( x\right) =e^{is\cdot x}$ is the restriction to $\Omega$ of the
complex exponential function corresponding to vector frequency $s=\left(
s_{1},\dots ,s_{d}\right) \in S$. Since a spectral measure is \emph{%
orthogonal}, i.e., $E\left( A\right) E\left( A^{\prime }\right) =0$ if $%
A\cap A^{\prime }=\varnothing $, it follows that%
\begin{equation*}
\left\{ \,e_{s}\mid s\in S\,\right\} 
\end{equation*}%
is an orthogonal basis for $L^{2}\left( \Omega\right) $; in other words, we say
that $\left( \Omega,S\right) $ is a \emph{spectral pair}.
\end{proof}

By the reasoning from the $d=1$ example, by analogy, one would expect that
the existence of commuting selfadjoint extensions will force $\Omega$ to tile $%
\mathbb{R}^{d}$ by translations, at least if $\Omega$ is also connected. And in
any case, one would expect that issues of spectrum and tile for bounded sets 
$\Omega$ in $\mathbb{R}^{d}$ are related.

Segal suggested to his U. of Chicago Ph.D. student David Shale (originally
from New Zealand) that he should take a closer look at when an open bounded
set $\Omega$ in $\mathbb{R}^{d}$ has commuting selfadjoint extensions for the $d$
formally hermitian, and formally commuting partial derivatives. Apparently
the research group around Irving Segal was very active, and everyone talked
to one another. At the time, Bent Fuglede was a junior professor in
Copenhagen, and he was thinking about related matters. Moreover, he quickly
showed that if you add an assumption, then the results for problems of
selfadjoint extensions do follow closely the simple case of the unit
interval, i.e., $d=1$.

Actually Fuglede had initially complained to Segal that the problem was
hard, and Segal suggested to add an assumption, assuming in addition that
the local action by translations in $\Omega$ is multiplicative. With this,
Fuglede showed that for special configuration of sets $\Omega$, there are
selfadjoint extensions, and that they are associated in a natural way with
lattices $L$ in $\mathbb{R}^{d}$. The spectrum of the representation $U$ is
a lattice $L$. By a \emph{lattice} we mean a rank-$d$ discrete additive
subgroup of $\mathbb{R}^{d}$.

Suppose now that $d$ commuting selfadjoint extensions exist for some bounded
open domain $\Omega$ in $\mathbb{R}^{d}$; and suppose in addition the
multiplicative condition is satisfied. When the spectral theorem of
Stone--Naimark--Ambrose--Godement (the SNAG theorem) is applied to a
particular choice of $d$ associated commuting unitary one-parameter groups,
Fuglede showed that $\Omega$ must then be a fundamental domain (also called a
tile for translations) for the lattice $L^{\ast }$ which is dual to $L$. If $%
L$ is a lattice in $\mathbb{R}^{d}$, the \emph{dual lattice} $L^{\ast }$ is%
\begin{equation*}
L^{\ast }=\left\{ \,\lambda \in \mathbb{R}^{d}\mid \lambda \cdot s\in 2\pi 
\mathbb{Z}\text{ for all }s\in L\,\right\} .
\end{equation*}%
But more importantly, Fuglede pointed out that the several-variable variant
of the spectral theorem (the SNAG theorem), and some potential theory, shown
that if in addition $\Omega$ is assumed connected, and has a `regular' boundary,
then the existence of commuting selfadjoint extensions implies that $%
L^{2}\left( \Omega\right) $ has an orthonormal basis of complex exponentials 
\begin{equation*}
\{\,\exp (is\cdot x)\mid s\text{ in some discrete subset }S\subset \mathbb{R}%
^{d}\}.
\end{equation*}%
In a later paper \cite{Jor82}, Jorgensen suggested that a pair of sets $(\Omega,S)
$ be called a \emph{spectral pair}, and that $S$ called a spectrum of $\Omega$.
With this terminology, we can state Fuglede's conjecture as follows: A
measurable subset $\Omega$ of $\mathbb{R}^{d}$ with finite positive Lebesgue
measure has a spectrum, i.e., is the first part of a spectral pair, if and
only if $\Omega$ tiles $\mathbb{R}^{d}$ with some set of translation vectors in $%
\mathbb{R}^{d}$.

Apparently Fuglede's work was all done around 1954, and was in many ways
motivated by von Neumann's thinking about unbounded operators.

In the mean time, David Shale left the problem, and he ended up writing a
very influential Ph.D. thesis on representations of the canonical
anticommutation relations in the case of an infinite number of degrees of
freedom. See \cite{Sha62}.

Irving Segal explained to me in 1976 how his suggestions from the early
fifties had been motivated by von Neumann. Von Neumann had hoped that all
the work he and Marshal Stone had put into the axiomatic formulation of the
spectral representation theorem would pay off on the central questions on
the theory of PDEs. Initially, Fuglede wasn't satisfied with what he had
proved, and he didn't get around to publishing his paper until 1974, see 
\cite{Fug74}, and only after being prompted by Segal, who had in the mean
time become the founding editor of the \emph{Journal of Functional Analysis}.

Fuglede had apparently felt that we needed to first understand non-trivial
examples that arise from tilings by translations with vectors that can have
irregular configurations, and aren't related in any direct way to a lattice.
Natural examples for $\Omega$ for $d=2$ that come to mind are the open interior
of the triangle or of the disk. But Fuglede proved in \cite{Fug74} that
these two planar sets do \emph{not} have spectra, i.e., they do not have the
basis property for any subset $S\subset \mathbb{R}^{2}$. Specifically, in
either case, there is no $S\subset \mathbb{R}^{2}$ such that $\left\{
\,e_{s}|_{\Omega}\mid s\in S\,\right\} $ is an orthogonal basis for $L^{2}\left(
\Omega\right) $. This of course is consistent with the spectrum-tile conjecture.

Very importantly, in his 1974 paper, Fuglede calculated a number of
instructive examples that showed the significance of combinatorics and of
finite cyclic groups in our understanding of spectrum and tilings; and he
made precise what is now referred to as the Fuglede conjecture (\cite{JoPe94,JoPe00,PeWa01,Ped04, Jor82}).

 Fuglede's question about equivalence of the two properties (existence of orthogonal Fourier frequencies for a given measurable subset $\Omega$ in $\br^d$) and the existence of a subset which makes $\Omega$ tile $\br^d$ by translations, was for $d = 1, 2$, and perhaps $3$. But of course the question is intriguing for any value of the dimension $d$. In recent years, a number of researchers, starting with Tao, have now produced examples in higher dimensions giving negative answers: By increasing the dimension, it is possible to construct geometric obstructions to tiling which do not have spectral theoretic counterparts; and vice versa. For some of these examples we refer the reader to \cite{FM06, Tao04, LaWa95, LaWa96a, LaWa96b,  LaWa96c, IKT03, Mat05, KL96}.

\begin{acknowledgements}
In addition to those mentioned by name in the paper, one or both of the co-authors had helpful conversations with W. B. Arveson,  T. Branson, J. M. G. Fell, A. Iosevich, R. V. Kadison, K. Kornelson, J. Lagarias, W. Lawton, G. Mackey, P. Muhly, S. Pedersen, R.S. Phillips, K. Shuman, R. Strichartz, and Y. Wang.
The authors are grateful to Steen Pedersen for a correction in an earlier version of Theorem \ref{ThmFJP}.
\end{acknowledgements}

\bibliographystyle{alpha}
\bibliography{fspw}
\end{document}